\documentclass[11pt]{amsart}
\usepackage{amssymb,mathrsfs,graphicx}
\usepackage{hyperref,color}

\def\bu{\mathbf u}
\def\bx{\mathbf x}
\def\by{\mathbf y}
\def\gtsim{\stackrel{{}_>}{{}_\sim}}

\title[A new model for flocking]{A new model for self-organized dynamics\\and its flocking behavior}

\author[Sebastien Motsch]{Sebastien Motsch}
\address[Sebastien Motsch]{\newline
  Department of Mathematics,  Institute for Physical Science and Technology\newline
  and \newline
  Center of Scientific Computation And Mathematical Modeling (CSCAMM)\newline
  University of Maryland, 
  College Park, MD 20742 USA}
\email[]{smotsch@cscamm.umd.edu}
\urladdr{http://www.seb-motsch.com}

\author[Eitan Tadmor]{Eitan Tadmor}
\address[Eitan Tadmor]{\newline
  Department of Mathematics,  Institute for Physical Science and Technology\newline
  and \newline
  Center of Scientific Computation And Mathematical Modeling (CSCAMM)\newline
  University of Maryland, 
  College Park, MD 20742 USA}
\email[]{tadmor@cscamm.umd.edu}
\urladdr{http://www.cscamm.umd.edu/tadmor}

\newtheorem{theorem}{Theorem}[section]
\newtheorem{lemma}[theorem]{Lemma}
\newtheorem{corollary}[theorem]{Corollary}
\newtheorem{proposition}[theorem]{Proposition}
\newtheorem{remark}[theorem]{Remark}

\newtheorem{definition}[theorem]{Definition}

{{\par\noindent \it Proof lemma. \nobreak}}%
{\nobreak \removelastskip \nobreak \it End proof lemma. \medbreak}
\newcommand{\ds}{\displaystyle}
\newcommand{\nb}{\nonumber}

\begin{document}

\date{\today}

\subjclass{92D25,74A25,76N10}
\keywords{Self-organized dynamics, flocking, active sets, kinetic formulation, moments, hydrodynamic formulation.}

\thanks{\textbf{Acknowledgment.} The work is supported by NSF grants DMS10-08397 and
  FRG07-57227.}

\dedicatory{To Claude Bardos on his $70^{th}$ birthday, with friendship and admiration}

\begin{abstract}
  We introduce a  model for self-organized dynamics which, we argue, addresses several drawbacks of the celebrated Cucker-Smale (C-S) model.  The proposed model does not only take into account the distance between agents, but instead, the influence between agents is scaled in term of their \emph{relative distance}. Consequently, our model does not involve any explicit dependence on the number of agents; only their geometry in phase space is taken into account. The use of relative distances destroys the symmetry property of the original C-S model, which was the key for the various recent studies of C-S flocking behavior. To this end, we introduce here a new framework to analyze the phenomenon of flocking for a rather general class of dynamical systems, which covers systems with \emph{non-symmetric} influence matrices. In particular, we analyze the flocking behavior of the proposed model as well as other strongly asymmetric models with ``leaders".

  The methodology presented in this paper, based on the notion of \emph{active sets}, carries over from  the particle to  kinetic and hydrodynamic descriptions. In particular, we discuss the hydrodynamic formulation of our proposed  model, and  prove  its unconditional flocking  for slowly decaying influence functions.
\end{abstract}
\maketitle
\centerline{\date}

\setcounter{tocdepth}{1}

\tableofcontents

\section{Introduction}
\setcounter{equation}{0}

The modeling of self-organized systems such as a flock of birds, a swarm of bacteria or a school of fish, \cite{aoki_simulation_1982,buhl_disorder_2006,camazine_self-organization_2001,grimm_individual-based_2005,huth_simulation_1992,huth_simulation_1994,parrish_self-organized_2002,viscido_individual_2004}, has brought new mathematical challenges. One of the many questions addressed concerns the emergent behavior in these systems and in particular, the emergence of ``flocking behavior''. Many models have been introduced to appraise the emergent behavior of self-organized systems \cite{ballerini_interaction_2008,birnir_ode_2007,couzin_collective_2002,gregoire_onset_2004,hemelrijk_self-organized_2008,reynolds_flocks_1987,vicsek_novel_1995,youseff_parallel_2008}.
The starting point for our discussion is the pioneering work of Cucker-Smale, \cite{cucker_emergent_2007,cucker_mathematics_2007}, which led to many subsequent studies \cite{birnir_ode_2007,carrillo_asymptotic_2010,ha_simple_2009,ha_particle_2008,ha_emergence_2009,shen_cucker-smale_2008}.  The C-S model describes how agents interact in order to align with their neighbors. It relies on a simple rule which goes back to \cite{reynolds_flocks_1987}: the closer two individuals are, the more they tend to align with each other (long range cohesion and short range repulsion are ignored). The motion of each agent ``$i$" is described by two quantities: its position, ${\bf x}_i\in \mathbb{R}^d$, and its velocity, ${\bf v}_i\in \mathbb{R}^d$. The evolution of each agent is then governed by the following dynamical system,
\begin{subequations}\label{eqs:CS}
  \begin{equation}
    \label{eq:CS}
    \frac{d{\bf x}_i}{dt} = {\bf v}_i, \qquad \frac{d{\bf v}_i}{dt} = \frac{\alpha}{N} \sum_{j=1}^N\, \phi_{ij} ({\bf v}_j-{\bf v}_i).
  \end{equation}
  Here, $\alpha$ is a positive constant and $\phi_{ij}$ quantifies the pairwise influence of agent ``$j$" on the alignment of agent ``$i$", as a function of their distance,
  \begin{equation}
    \label{eq:kernel_sym}
    \phi_{ij} := \phi(|{\bf x}_j-{\bf x}_i|).
  \end{equation}
\end{subequations}
The so-called \emph{influence function}, $\phi(\cdot)$, is a strictly positive decreasing function which, by rescaling $\alpha$ if necessary, is normalized so that $\phi(0)=1$. A prototype example for such an influence function is given by $\phi(r) = (1+r)^{-s}, s>0$. Observe that the C-S model (\ref{eqs:CS}) is \emph{symmetric} in the sense that the coefficients matrix $\phi_{ij}$ is, namely, agents ``$i$" and ``$j$" have the same influence on the alignment of each other,
\begin{equation}
  \label{eq:sym_phi_ij}
  \phi_{ij} = \phi_{ji}.
\end{equation}
The symmetry in the C-S model is the cornerstone for studying the long time behavior of (\ref{eqs:CS}). Indeed, symmetry implies that the total momentum in the C-S model is conserved,
\begin{subequations}\label{eqs:CSdyn}
  \begin{equation}
    \label{eq:mom}
    \frac{d}{dt} \left(\frac{1}{N} \sum_{i=1}^N {\bf v}_i(t)\right) =0 \ \;\;  \mapsto \;\; \
    \overline{\bf v}(t):=\frac{1}{N} \sum_{i=1}^N {\bf v}_i(t) =\overline{\bf v}(0).
  \end{equation}
  Moreover, the symmetry of (\ref{eq:sym_phi_ij}) implies that the C-S system is \emph{dissipative},
  \begin{equation}
    \label{eq:dissip}
    \frac{d}{dt} \frac{1}{N} \sum_{i} |{\bf v}_i-\overline{{\bf v}}|^2 = -\frac{\alpha}{2N^2} \sum_{i,j} \phi_{ij}  |{\bf v}_i-{\bf v}_j|^2 \leq  - \min_{ij}\phi_{ij}(t) \times \frac{\alpha}{N} \sum_{i}  |{\bf v}_i-\overline{{\bf v}}|^2.
  \end{equation}
\end{subequations}
Consequently, (\ref{eqs:CSdyn}) yields the large time behavior,  $\bx_i(t) \approx \overline{\bf v}t$, and hence $\min_{ij}\phi_{ij}(t) \gtsim  \phi(|\overline{\bf v}|t)$. This, in turn, implies that  the  C-S dynamics converges to the bulk mean velocity, 
\begin{equation}\label{eq:CSbulk}
  {\bf v}_i(t) \stackrel{t\rightarrow \infty}{\longrightarrow}  \overline{\bf v}(0),
\end{equation}
provided the long-range influence between agents, $\phi(|{\bf x}_j-{\bf x}_i|)$, decays sufficiently slow in the sense that $\phi(\cdot)$ has a diverging tail,
\begin{equation}
  \label{eq:phi_not_integrable}
  \int^\infty \phi(r)\,dr=\infty.
\end{equation}
We conclude that the C-S model with a slowly decaying influence function (\ref{eq:phi_not_integrable}), has an \emph{unconditional} convergence to a so-called \emph{flocking dynamics}, in the sense that (i) the diameter, $\max_{i,j}|\bx_i(t)-\bx_j(t)|$, remains uniformly bounded, thus defining the domain of the ``flock"; and (ii) all agents of this flock will approach the same velocity --- the emerging ``flocking velocity".
\begin{definition}\cite[p. 416]{ha_particle_2008}\label{defi:flock}
  Let $\{{\bf x}_i(t) ,{\bf v}_i(t) \}_{i= 1,\dots,N}$ be a given particle system, and let $d_X(t)$ and $d_V(t)$ denote its diameters in position and velocity phase spaces,
  \begin{subequations}
    \begin{eqnarray}
      \label{eq:d_X}
      d_X(t) &=& \max_{i,j} |{\bf x}_j(t)-{\bf x}_i(t)|,\\
      \label{eq:d_V}
      d_V(t) &=& \max_{i,j} |{\bf v}_j(t)-{\bf v}_i(t)|.
    \end{eqnarray}
  \end{subequations}
  The  system $\{{\bf x}_i(t) ,{\bf v}_i(t) \}_{i= 1,\dots,N}$ is said to converge to a flock if 
  the following two conditions hold, uniformly in $N$,
  \begin{equation}\label{eq:cond_flock}
    \sup_{t\geq0}\, d_X(t) < +\infty \quad \text{ and } \quad \lim_{t\rightarrow+\infty} d_V(t) =0.
  \end{equation}
\end{definition}

\begin{remark}
  One can distinguish between two types of flocking behaviors. When \eqref{eq:cond_flock} holds for \emph{all} initial data, $\{{\bf x}_i(0) ,{\bf v}_i(0) \}_{i= 1,\dots,N}$, it is referred to as \emph{unconditional} flocking, e.g., \cite{carrillo_asymptotic_2010,cucker_emergent_2007,ha_simple_2009, ha_particle_2008,shen_cucker-smale_2008}. In contrast, \emph{conditional} flocking occurs when \eqref{eq:cond_flock} is limited to a certain class of initial configurations.
\end{remark}

The flocking behavior of the C-S model derived in \cite{ha_particle_2008} was based on the $\ell^2$-based arguments outlined in  (\ref{eqs:CSdyn}). Other approaches, based on spectral analysis, $\ell^1$- and $\ell^\infty$-based estimates  were used in \cite{carrillo_asymptotic_2010,cucker_emergent_2007,ha_simple_2009} to derive C-S flocking with a (refined version of) slowly decaying influence function (\ref{eq:phi_not_integrable}). Though the derivations are different, they all require the symmetry of the C-S influence matrix, $\phi_{ij}$.

Despite the elegance of the results regarding its flocking behavior, the description of self-organized dynamics by the C-S model suffers from several drawbacks. We mention in this context the normalization of C-S model in (\ref{eq:CS}) by the total number of agents, $N$, which is shown, in section \ref{sec:draw} below, to be inadequate for far-from-equilibrium scenarios. The first main objective of this work is to introduce a new model for self-organized dynamics which, we argue, will address several drawbacks of the C-S model.  Indeed, the model introduced in section \ref{sec:newmodel} below, does not just take into account the distance between agents, but instead, the influence two agents exert on each other is scaled in term of their \emph{relative distances}. As a consequence, the proposed model does not involve any explicit dependence on the number of agents --- just their geometry in phase space is taken into account. It lacks, however, the symmetry property of the original C-S model, (\ref{eq:sym_phi_ij}). This brings us to the second main objective of this work: in section \ref{sec:tools} we develop a new framework to analyze the phenomenon of flocking for a rather general class of dynamical systems of the form,
\[
\frac{d{\bf x}_i}{dt} = {\bf v}_i, \qquad \frac{d{\bf v}_i}{dt} = \alpha 
\sum_{j\ne i}^N a_{ij}({\bf v}_j-{\bf v}_i), \quad \; a_{ij} \geq 0, \; \ \sum_{j\neq i}a_{ij}<1,
\]
which allows for non-symmetric influence matrices, $a_{ij}\ne a_{ji}$. Here we utilize the concept of \emph{active sets}, which enables us to define the notion of a neighborhood of an agent; this quantifies the ``neighboring" agents in terms of their level of influence, rather than the usual Euclidean distance. The cornerstone of our study of flocking behavior, presented in section \ref{sec:lemma}, is based on a key algebraic lemma, interesting for its own sake, which bounds the maximal action of antisymmetric matrices on active sets. Accordingly, the main result summarized in theorem \ref{thm:ine_dx_dv}, quantifies the dynamics of the diameters, $d_X(t)$ and $d_V(t)$, in terms of the global active set associated with the model.  We conclude, in section \ref{sec:new-model-flock}, that the dynamics of our proposed model will experience unconditional flocking provided the influence function $\phi$ decays sufficiently slowly such that,
\begin{equation}
  \label{eq:phi2_not_integrable}
  \int^\infty \phi^2(r)\,dr=\infty.
\end{equation}
This is slightly more restrictive than the condition for flocking in the symmetric case of C-S model, (\ref{eq:phi_not_integrable}).  Another fundamental difference between the flocking behavior of these two models is pointed out in remark \ref{rem:differences} below: unlike the C-S flocking to the initial bulk velocity $\overline{\bf v}(0)$ in (\ref{eq:CSbulk}), the asymptotic flocking velocity of our proposed model is not necessarily encoded in the initial configuration as an invariant of the dynamics, but it is \emph{emerging} through the flocking dynamics of the model.

The methodology developed in this work is not limited to the new model, whose flocking behavior is analyzed in section \ref{subsec:new-model-flock}. In section \ref{sec:leader}, we use the concept of active sets to study the flocking behavior of models with a ``leader''. Such models are strongly asymmetric, since they assume that some individuals are more influential than the others.

Finally, in section \ref{sec:meso} and, respectively, section \ref{sec:hydro}, we pass from the particle to  kinetic and, respectively, hydrodynamic descriptions of the proposed model. The latter amounts to the usual statements of conservation of mass, $\rho$, and balance of momentum, $\rho\bu$,
\begin{subequations}\label{eqs:hydro_cons}
  \begin{align} 
    \partial_t\rho+ \nabla_{\bx}\cdot(\rho\bu)&=0\\
    \label{eq:hydro_con2}
    \partial_t(\rho \bu)+ \nabla_\bx(\rho\bu\otimes\bu)&=\alpha\rho\left(\frac{\langle {\bf u}\rangle}{\langle 1\rangle}-\bu\right),\quad
    \langle {\bf w}\rangle(\bx):=\int_{\by} \phi(|\bx-\by|){\bf w}(\by)\rho(\by)d\by.
  \end{align}
\end{subequations}
We extend our methodology of active sets to study the flocking behavior in these contexts of mesoscopic  and macroscopic scales. In particular, we prove  the unconditional flocking behavior of (\ref{eqs:hydro_cons})  with a slowly decaying influence function, $\phi$, such that (\ref{eq:phi2_not_integrable}) holds,
\[
\sup_{{\bf x}, {\bf y}\in Supp\left(\rho(\cdot,t)\right)}|{\bf u}(t,{\bf x})-{\bf u}(t,{\bf y})| \stackrel{t\rightarrow \infty}{\longrightarrow} 0.
\]

\section{A model for self-organized dynamics}
\setcounter{equation}{0}

In this section, we introduce the new model that will be the core of this work. This model is motivated by some drawbacks of the C-S model.

\subsection{Drawbacks of the C-S model}\label{sec:draw}
Originally, the C-S model was introduced in \cite{cucker_emergent_2007} to model a finite number of agents. The normalization pre-factor ${1}/{N}$ in (\ref{eq:CS}) was added later in Ha and Tadmor, \cite{ha_particle_2008}, in order to study the ``mean-field'' limit as the number of agents $N$ becomes very large. This modification, however, has a drawback in the modeling: the motion of an agent is modified by the total number of agents even if its dynamics is only influenced by essentially a few nearby agents. To better explain this problem, we sketch a particular scenario shown in figure \ref{fig:sketch_birds}. Assume that there is a small group of $N_1$ agents, $G_1$, at a large distance from a large group of $N_2$ agents, $G_2$; by assumption, we have $N_1<<N_2$. If the distance between the two groups is large enough, we have,
\begin{subequations}\label{eqs:sketch_birds}
  \begin{equation}
    \label{eq:no_interaction}
    \phi_{ij} \approx 0 \qquad \text{if } i\in G_1 \text{ and } j \in G_2.
  \end{equation}
  In this situation, the C-S dynamics of  every agent ``$i$" in group $G_1$ reads,
  \begin{equation}
    \frac{d{\bf v}_i}{dt} \approx \frac{\alpha}{N_1+N_2} \sum_{1\leq j\leq N_1} \phi_{ij} ({\bf v}_j-{\bf v}_i), \qquad i\in G_1.
  \end{equation}
\end{subequations}
Therefore, since  there are only $N_1$ ``essentially" active neighbors of ``$i$", yet we  average over the much larger set of  $N_1+N_2 \gg N_1$ agents, we would have ${d{\bf v}_i}/{dt} \approx 0$. Thus,  the presence of a large group of agents $G_2$ in the horizon of $G_1$, will almost halt the dynamics of $G_1$.

\begin{figure}[ht]
  \centering
  \includegraphics[scale=.65]{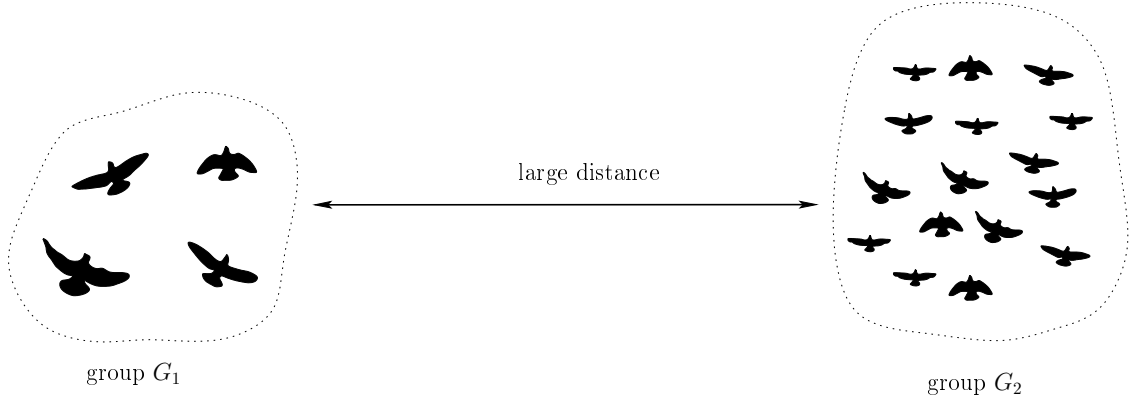}
  \caption{A small group of birds $G_1$ at a large distance from a larger group $G_2$
    (\ref{eq:no_interaction}). Due to the normalization $1/N$ in the C-S
    model (\ref{eq:CS}), the group $G_1$ will almost stop interacting.}
  \label{fig:sketch_birds}
\end{figure}

\subsection{A  model with non-homogeneous phase space}\label{sec:newmodel}
We propose the following dynamical system to describe the motion of  agents 
$\{{\bf x}_i(t),{\bf v}_i(t) \}_{i=1}^N$,
\begin{equation}
  \label{eq:eitan}
  \frac{d{\bf x}_i}{dt} = {\bf v}_i, \qquad \frac{d{\bf v}_i}{dt} = \frac{\alpha}{\sum_{k=1}^N \phi_{ik}}
  \sum_{j=1}^N\, \phi_{ij}\, ({\bf v}_j-{\bf v}_i), \qquad \phi_{ij}=\phi(|\bx_j-\bx_i|).
\end{equation}
Here, $\alpha$ is a positive constant and $\phi(\cdot)$ is the influence function. The main feature here is that the  influence  agent ``$j$" has on the alignment of agent ``$i$",  is weighted by the total influence, $\sum_{k=1}^N \phi_{ik}$, exerted on  agent ``$i$".

In the case where all agents are clustered around the same distance,
i.e., $\phi_{ij} \approx \phi_0$, then the model (\ref{eq:eitan}) amounts to C-S dynamics,
\begin{displaymath}
  \frac{d{\bf v}_i}{dt} = \frac{\alpha}{N\phi_0} \sum_{j=1}^N\, \phi_{ij}\, ({\bf v}_j-{\bf v}_i).
\end{displaymath}
But unlike the C-S model, the space modeled by (\ref{eq:eitan}) need \emph{not} be homogeneous. In particular, it better captures strongly non-homogeneous scenarios such as those depicted in \ref{eqs:sketch_birds}: the motion of an agent ``$i$" in the smaller group $G_1$ will be, to a good approximation, dominated by the agents in group $G_1$,
\begin{displaymath}
  \frac{d{\bf v}_i}{dt} \approx \frac{\alpha}{N_1 \phi_0} \sum_{1\leq j \leq N_1}\, \phi_{ij}\, ({\bf v}_j-{\bf v}_i).
\end{displaymath}
Here, $\phi_0$ is the coefficient of interaction inside the nearby group $G_1$, i.e., $\phi_{ij} \approx \phi_0$ for $i,\,j\,\in G_1$, whereas the agents in the ``remote" group $G_2$, will only have a negligible influence, $\sum_k\phi_{ik} \approx N_1\phi_0$.

The  normalization of pairwise interaction  between agents in terms of \emph{relative} influence has the consequence of loss of symmetry: the model (\ref{eq:eitan}) can be written as,
\[
\frac{d{\bf x}_i}{dt} = {\bf v}_i, \qquad \frac{d{\bf v}_i}{dt} = \alpha \sum_{j=1}^N\, a_{ij} ({\bf v}_j-{\bf v}_i),
\]
where the coefficients $a_{ij}$, given by,
\[
a_{ij} = \frac{\phi(|{\bf x}_j-{\bf x}_i|)}{\sum_{k=1}^N \phi(|{\bf x}_k-{\bf x}_i|)},
\]
lack the symmetry property, $a_{ij} \neq a_{ji}$. Two more examples of models with asymmetric influence matrices  will be discussed below. The flocking behavior of a model with \emph{leaders}, in which agents follow one or more ``influential" agents and hence lack symmetry,  is analyzed  in section \ref{sec:new-model-flock} below.  In section \ref{sec:conc} we introduce a model with \emph{vision} in which agents are aligned with those agents ahead of them, as another prototypical example for self-organized dynamics which lacks symmetry, and we comment on the difficulties in its flocking analysis. Tools for studying flocking behavior of  such asymmetric models are outlined in the next section.

\section{New tools to study flocking}\label{sec:tools}
\setcounter{equation}{0}

We want to study the long time behavior of the proposed model (\ref{eq:eitan}). The lack of symmetry, however, breaks down the nice properties of conservation of momentum, (\ref{eq:mom}), and energy dissipation, (\ref{eq:dissip}), we had with the C-S model. The main tool for studying the C-S flocking was the variance, $\left(\sum |{\bf v}_i-\overline{\bf v}|^p\right)^{1/p}$, in either one of its $\ell^p$-versions, $p=1,2$ or $p=\infty$. But since the momentum is not conserved in the proposed model (\ref{eq:eitan}), the variance is no longer a useful quantity to look at; indeed, it is not even a priori clear what should be the ``bulk" velocity, $\overline{\bf v}$, to measure such a variance.

In this section, we discuss the tools to study the flocking behavior for a rather general class of dynamical systems of the form,
\begin{subequations}\label{eqs:eitan_average}
  \begin{equation}  \label{eq:eitan_averagea}
    \frac{d{\bf x}_i}{dt} = {\bf v}_i, \qquad \frac{d{\bf v}_i}{dt} = \alpha 
    \sum_{j\ne i}^N a_{ij}({\bf v}_j-{\bf v}_i), \quad a_{ij} \geq 0.
  \end{equation}
  Here, $\alpha$ is a positive constant, and $a_{ij}>0$ quantifies the pairwise influence of agent ``$j$" on the alignment of agent ``$i$", through possible dependence on the state variables, $\{\bx_k,{\bf v}_k\}_k$.  By rescaling $\alpha$ if necessary, we may assume without loss of generality that the $a_{ij}$'s are normalized so that
  \begin{equation}\label{eq:eitan_averageb}
    \sum_{j\neq i} a_{ij} \leq 1.
  \end{equation}
\end{subequations}
Setting $a_{ii}:=1-\sum_{j\ne i} a_{ij}$, we can rewrite (\ref{eqs:eitan_average}) in the form
\begin{subequations}\label{eqs:average}
  \begin{equation}\label{eq:averagea}
    \frac{d{\bf x}_i}{dt} = {\bf v}_i, \qquad \frac{d{\bf v}_i}{dt} = \alpha 
    (\overline{\bf v}_i-{\bf v}_i), 
  \end{equation}
  where the average velocity, $\overline{{\bf v}}_i$, is given by a \emph{convex combination} of the velocities surrounding  agent ``$i$",
  \begin{equation}\label{eq:averageb}
    \overline{{\bf v}}_i := \sum_{j=1}^N a_{ij} {\bf v}_j, \qquad \sum_j a_{ij}=1.
  \end{equation}
\end{subequations}
We should emphasize that there is no requirement of symmetry,  allowing $a_{ij}\ne a_{ji}$. This setup includes, in particular,  the  model for self-organized dynamics proposed in (\ref{eq:eitan}), with asymmetric coefficients $a_{ij}=\phi_{ij}/\sum_k \phi_{ik}$.

In order to study the flocking behavior of (\ref{eqs:eitan_average}), we quantify in section \ref{sec:active}, the \emph{decay} of the diameter, $d_V(t)$ using the notion of \emph{active sets}.  The relevance of this concept of active sets is motivated by a key lemma on the maximal action of antisymmetric matrices outlined in section \ref{sec:lemma}. This, in turn, leads to the main estimate of theorem \ref{thm:ine_dx_dv}, which governs the evolution of $d_X(t)$ and $d_V(t)$.

\subsection{Maximal action of antisymmetric matrices}\label{sec:lemma}
We begin our discussion with the following key lemma.        

\begin{lemma}
  \label{lem:fundamental_lemma}
  Let $S$ be an antisymmetric matrix, $S_{ij}=-S_{ji}$ with maximal entry $|S_{ij}|\leq M$.  Let $u,w \in \mathbb{R}^N$ be two given real vectors with positive entries, $u_i, w_i\geq 0$, and let $\overline{U}$, $\overline{W}$ denote their respective sum, $\overline{U}=\sum_i u_i$ and $\overline{W}=\sum_j w_j$. Fix $\theta>0$ and let $\lambda(\theta)$ denote the number of ``active entries'' of $u$ and $w$ at level $\theta$, in the sense that,
  \begin{displaymath}
    \lambda(\theta) =|\Lambda(\theta)|, \qquad \Lambda(\theta):=\left\{ j\ \big| \ u_j \geq \theta\, \overline{U} \, \, \text{and} \,
      \, w_j \geq \theta\, \overline{W} \ \right\}.
  \end{displaymath}
  Then, for every $\theta>0$, we have
  \begin{equation}
    \label{eq:fundamental_inequality}
    |\langle Su,w\rangle| \leq M \overline{U}\, \overline{W}\left(1- \lambda^2(\theta)\,\theta^2 \right).
  \end{equation}
\end{lemma}

\begin{remark} Lemma \ref{lem:fundamental_lemma} tells us that the maximal action of $S$ on $u,w$, does not exceed
  \[
  |\langle Su,w\rangle| \leq M \overline{U}\, \overline{W}\min_\theta\left(1- \lambda^2(\theta)\,\theta^2 \right).
  \]
  which improves the obvious upper-bound, $|\langle Su,w\rangle| \leq M \overline{U}\, \overline{W}$.
\end{remark}


\begin{proof}
  Using the antisymmetry of $S$, we find
  \begin{displaymath}
    \langle Su,w\rangle = \sum_{i,j} S_{ij}u_iw_j = \frac{1}{2} \sum_{i,j} S_{ij}\big(u_iw_j - u_jw_i\big),
  \end{displaymath}
  and since $S$ is bounded by $M$, we obtain the inequality,
  \begin{displaymath}
    |\langle Su,w\rangle| \leq \frac{M}{2} \sum_{i,j} |u_iw_j - u_jw_i|.
  \end{displaymath}
  The identity,  $|a-b|\equiv  a+b -2\min(a,b)$ for $a,b \geq 0$, implies
  \begin{eqnarray}
    |\langle Su,w\rangle|  & \leq &  \frac{M}{2}  \sum_{i,j} \Big(u_iw_j + u_j w_i -
    2\min\{u_iw_j,u_jw_i\}\Big) \nonumber \\
    \label{eq:ine_fund_temp}
    &= & M\Big(\overline{U}\,\overline{W} -  \sum_{i,j} \min\{u_iw_j,u_jw_i\}\Big).
  \end{eqnarray}
  By assumption, there are at least $\lambda(\theta)$ active entries at level $\theta$ which  satisfy \emph{both} inequalities,
  \begin{displaymath}
    k\in \Lambda(\theta): \ \ u_k \geq \theta \,\overline{U} \ \text{and} \  w_k \geq \theta \,\overline{W}.
  \end{displaymath}
  Therefore, by restricting  the sum in (\ref{eq:ine_fund_temp}) only to the pairs of these active entries we find
  \begin{displaymath}
    |\langle Su,w\rangle| \leq  M\Big(\overline{U}\,\overline{W} -  \sum_{i,j\in \Lambda(\theta)} \min\{u_iw_j,u_jw_i\}\Big) \leq   M\left(\overline{U}\,\overline{W}  \,-\, \lambda^2(\theta) \cdot \theta\,
      \overline{U} \cdot \theta\, \overline{W}\right),
  \end{displaymath}
  and the desired inequality (\ref{eq:fundamental_inequality}) follows.
\end{proof}

\subsection{Active sets and the decay of diameters}\label{sec:active}
The  concept of an \emph{active set} aims to determine a neighborhood of one or more  agents in (\ref{eqs:eitan_average}) based on the so-called \emph{influence matrix}, $\{a_{ij}\}$, rather than the usual Euclidean distance. The following definition, which applies to arbitrary matrices, is formulated in the language of influence matrices.

\begin{definition}[Active sets]\label{defi:active_set}
  Let $\{a_{ij}\}$ be a normalized influence matrix, $a_{ij}>0, \ \sum_j a_{ij}=1$.  Fix $\theta>0$. The active set, $\Lambda_{p}(\theta)$, is the set of agents which influence ``$p$" more than $\theta$,
  \begin{equation}    \label{eq:active_set_p}
    \Lambda_{p}(\theta) := \{ j \ \big| \ a_{pj} \geq \theta\}.
  \end{equation}
  The global active set, $\Lambda(\theta)$, is the intersection of all the active sets at that level,
  \begin{equation}
    \label{eq:active_set_global}
    \Lambda(\theta) = \bigcap_{p} \Lambda_{p}(\theta). 
  \end{equation}
\end{definition}

This notion of active set, $\Lambda_p(\theta)$, defines a ``neighborhood"   for agent ``$p$", and can be generalized to more than just one agent. For example, 
\begin{equation}\label{eq:active_set_pq}
  \Lambda_{pq}(\theta):=\Lambda_p(\theta)\cap\Lambda_q(\theta),
\end{equation}
is the set of all agents whose influence on \emph{both}, ``$p$" and ``$q$", is larger than $\theta$, see figure \ref{fig:active_set}.

The number of agents in an active set $\Lambda_{\mathcal I}(\theta)$ is denoted by $\lambda_{\mathcal I}(\theta)$, e.g. $\lambda_{pq}(\theta) = |\Lambda_{pq}(\theta)|.$ The numbers $\{\lambda_{pq}(\theta)\}_{pq}$ are difficult to compute for general $\theta$'s: one needs to \emph{count} the number of pairs of agents in the underlying graph ${\mathcal G}$, which stay connected above level $\theta$,
\begin{equation}\label{eq:graph}
  {\mathcal G}_{i,j} = \left\{
    \begin{array}{ll}
      1 & \text{ if agent ``$i$" is influenced by ``$j$"}: a_{ij}>0,\\
      0 & \text{otherwise}.
    \end{array}
  \right.
\end{equation}
One simple case we can count, however, occurs when $\theta$ takes the minimal value, $\theta= \min_{ij}a_{ij}$. Then, the active sets $\Lambda_{p}(\theta)$ includes all the agents, $\Lambda_{p}(\theta)_{\theta=\min_{ij}a_{ij}} = \{1,\dots,\,N\}$, and since this applies for every "$p$", then $\Lambda_{pq}(\theta)$ and the global active set, $\Lambda(\theta)$, include all agents,
\begin{equation}   \label{eq:theta_1}
  \lambda(\theta)_{\displaystyle |\theta=\!\min_{ij} a_{ij}} \,=\, N.
\end{equation}

\begin{figure}[ht]
  \centering
  \includegraphics[scale=1]{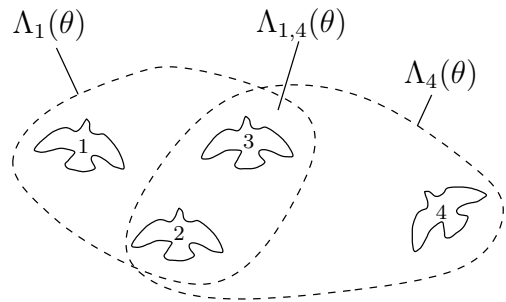}
  \caption{An illustration of  active sets. Here,
    $\Lambda_1(\theta) = \{1,\,2,\,3\}$ and $\Lambda_4(\theta)=\{2,\,3,\,4\}$. The pairwise active set, $\Lambda_{14}=\Lambda_1(\theta)\cap\Lambda_4(\theta)$, consists of  agents ``$2$" and ``$3$".}
  \label{fig:active_set}
\end{figure}

Armed with the notion of active set and with the key lemma \ref{lem:fundamental_lemma} on maximal action of antisymmetric matrices, we can now state our main result, measuring the decay of the diameters $d_X(t)$ and $d_V(t)$ in the dynamical system (\ref{eqs:average}).

\begin{theorem}
  \label{thm:ine_dx_dv}
  Let $\{{\bf x}_i(t),{\bf v}_i(t)\}_{i}$ be a solution of the dynamical system \eqref{eqs:average}. Fix an arbitrary $\theta>0$ and let $\lambda_{pq}(\theta)$ be the number of agents in the active sets, $\Lambda_{pq}(\theta)$, associated with the influence matrix of \eqref{eqs:average}.  Then the diameters of this solution, $d_X(t)$ and $d_V(t)$, satisfy,
  \begin{subequations}\label{eqs:evo_d_XV}  
    \begin{eqnarray}
      \label{eq:evo_d_X}
      \frac{d}{dt} d_X(t) & \leq & d_V(t) \\
      \label{eq:evo_d_V}
      \frac{d}{dt} d_V(t) & \leq & -\alpha\,\min_{pq}\lambda_{pq}^2(\theta)\,\theta^2 \,d_V(t).
    \end{eqnarray}
  \end{subequations}
\end{theorem}

Since $\Lambda(\theta) \subset \Lambda_{p,q}(\theta)$ then $\lambda_{pq}(\theta)\geq \lambda(\theta)$ and (\ref{eq:evo_d_V}) yields the following global version of the theorem above.
\begin{theorem}\label{thm:ine_dx_dv_global}
  Fix an arbitrary $\theta>0$ and let $\lambda(\theta)$ be the number of agents in the global active set, $\Lambda(\theta)$, associated with \eqref{eqs:average}.  Then the diameters of its solution, $d_X(t)$ and $d_V(t)$, satisfy,
  \begin{subequations}\label{eqs:evo_d_XV_global}  
    \begin{eqnarray}
      \label{eq:evo_d_X_global}
      \frac{d}{dt} d_X(t) & \leq & d_V(t) \\
      \label{eq:evo_d_V_global}
      \frac{d}{dt} d_V(t) & \leq & -\alpha\,\lambda^2(\theta)\,\theta^2 \,d_V(t).
    \end{eqnarray}
  \end{subequations}
\end{theorem}

\begin{proof}[Proof of theorem \ref{thm:ine_dx_dv}]
  We fix our attention to two trajectories ${\bf x}_p(t)$ and ${\bf x}_q(t)$, where $p$ and $q$ will be determined later. Their relative distance satisfies,
  \begin{equation*}
    \frac{d}{dt} |{\bf x}_p-{\bf x}_q|^2 = 2 \langle {\bf x}_p\!-\!{\bf x}_q,{\bf
      v}_p\!-\!{\bf v}_q\rangle  \;\leq \, 2 |{\bf x}_p\!-\!{\bf x}_q| |{\bf v}_p\!-\!{\bf v}_q|,
  \end{equation*}
  which implies,
  \begin{equation*}
    \frac{d}{dt} |{\bf x}_p(t)-{\bf x}_q(t)| \leq d_V(t).
  \end{equation*}
  Thus, (\ref{eq:evo_d_X}) holds. Next, we turn to study the corresponding relative distance in velocity phase space,
  \begin{eqnarray}\label{eq:xyz}
    \frac{d}{dt}\,|{\bf v}_p-{\bf v}_q|^2 &=& 2\alpha\,\langle {\bf v}_{p} \!-\! {\bf v}_{q}\,, \dot{{\bf v}}_{p} \!-\! \dot{{\bf v}}_{q}\rangle  \\
    &=& 2\alpha\langle {\bf v}_{p} \!-\! {\bf v}_{q}\,, \overline{{\bf v}}_{p} \!-\! \overline{{\bf v}}_{q}\rangle  \,- \;2\alpha\,|{\bf v}_{p} - {\bf v}_{q}|^2;\nonumber
  \end{eqnarray}
  recall that ${\bf \overline{v}}_p$ and ${\bf \overline{v}}_p$ are the average velocities defined in (\ref{eq:eitan_averageb}). Given that $\sum_\ell a_{k\ell}\equiv 1$, the difference of these averages is given by,
  \begin{eqnarray*}
    \overline{{\bf v}}_p-\overline{{\bf v}}_q & = &  \sum_j a_{pj} {\bf v}_j-
    \overline{{\bf v}}_q = \sum_j a_{pj} \big({\bf v}_j-
    \overline{{\bf v}}_q\big) \\
    & = &  \sum_j a_{pj} \Big({\bf v}_j- \sum_i
    a_{qi} {\bf v}_i\Big) =  \sum_j \sum_i a_{pj}
    a_{qi}\big({\bf v}_j- {\bf v}_i\big).
  \end{eqnarray*}
  Inserting this into (\ref{eq:xyz}), we find,
  \begin{equation}
    \label{eq:vp_vq}
    \frac{d}{dt}\,|{\bf v}_p-{\bf v}_q|^2 = 2\alpha \left(\sum_{ij} a_{pj} a_{qi} \langle {\bf
        v}_p\!-\!{\bf v}_q,\,{\bf v}_j\!-\!{\bf v}_i\rangle  \;\;-\;\;|{\bf v}_{p} - {\bf
        v}_{q}|^2\right).
  \end{equation}
  To upper-bound the first quantity on the right, we use the lemma \ref{lem:fundamental_lemma} with $u_i = a_{pi}$, $w_i=a_{qi}$ and the antisymmetric matrix \mbox{$S_{i,j}= \langle {\bf v}_p\!-\!{\bf v}_q,\,{\bf v}_j\!-\!{\bf v}_i\rangle$}: since $|S_{i,j}|\leq d_V^2$, $\overline{U} = \sum_i u_i=1$ and $\overline{W}= \sum_i w_i =1$, we deduce,
  \begin{displaymath}
    \Big|\sum_{ij} a_{pj} a_{qi} \langle {\bf v}_p\!-\!{\bf v}_q,\,{\bf v}_j\!-\!{\bf v}_i\rangle \Big| \leq d_V^2 (1 - \lambda_{pq}^2(\theta)\, \theta^2).
  \end{displaymath}
  Here, $\lambda_{pq}(\theta)$ is the number of agents in the active set $\Lambda_{pq}(\theta)$,
  \begin{displaymath}
    \lambda_{pq}(\theta) = |\{ j \ \big| \ a_{pj} \geq \theta \text{ and } a_{qj} \geq \theta\}|.
  \end{displaymath}
  Therefore, the relative velocity ${\bf v}_p-{\bf v}_q$ in (\ref{eq:vp_vq}) satisfies,
  \begin{displaymath}
    \frac{d}{dt}\,|{\bf v}_p-{\bf v}_q|^2 \leq 2\alpha \Big(d_V^2 (1 - \lambda_{pq}^2(\theta)\, \theta^2) \;\;-\;\;|{\bf v}_{p} - {\bf v}_{q}|^2\Big).
  \end{displaymath}
  In particular, if we choose $p$ and $q$ such that $ |{\bf v}_p(t)-{\bf v}_q(t)|=d_V(t)$,
  the last inequality reads,
  \begin{equation}
    \label{eq:evo_d_V_temp}
    \frac{d}{dt}\,d_V^2(t) \leq -2\alpha\,\min_{pq}\lambda_{pq}^2(\theta)\,\theta^2 \,d_V^2(t).
  \end{equation}
  and the inequality (\ref{eq:evo_d_V}) follows.
\end{proof}

\begin{remark}\label{rem:v_is_bounded}
  Equation \eqref{eq:evo_d_V} tells us that the diameter in velocity phase space, $d_V(t)$, is decreasing in time. In fact, an even stronger statement holds, namely, if we let $\Omega(t)$ denote the convex hull of the velocities, $\Omega(t) := \text{Conv}\left(\{{\bf v}_i(t)\}_{i=1 \dots N}\right)$, then $\Omega(t)$ is decreasing in time in the sense of set inclusion,
  \begin{equation}
    \label{eq:Omega_decreasing}
    \Omega(t_1) \supset \Omega(t_2) \; \text{ if } t_1\leq t_2.
  \end{equation}
  Indeed, by convexity, $\overline{{\bf v}}_i \in \Omega(t)$ for any $i$, and consequently, if ${\bf v}_i$ is at the frontier of $\Omega$, then the vector $(\overline{{\bf v}}_i - {\bf v}_i)$ points to the interior of $\Omega$ at ${\bf v}_i$, see figure \ref{fig:normal_vector}. More precisely, if ${\bf n}$ denotes the outward-pointing normal to $\Omega$ at ${\bf v}_i$, then $\dot{{\bf v}}_i \cdot {\bf n} = (\overline{{\bf v}}_i-{\bf v}_i)\cdot{\bf n} \leq 0$ Therefore, the frontier of $\Omega(t)$ is a ``fence'' \cite{hubbard_differential_1995} for the vectors ${\bf v}_i(t)$ and \eqref{eq:Omega_decreasing} follows.

  \begin{figure}[ht]
    \centering
    \includegraphics[scale=1]{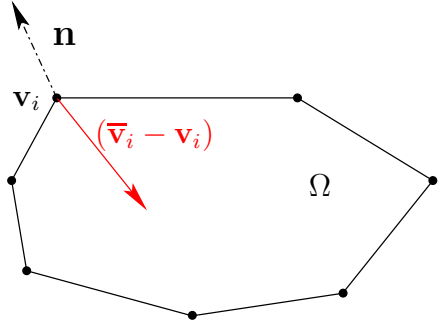}
    \caption{At the frontier of the convex hull $\Omega$, the vector
      $(\overline{{\bf v}}_i-{\bf v}_i)$ points to the interior of $\Omega$. Thus, for any outward-pointing
      normal vector ${\bf n}$ at ${\bf v}_i$, we have: $\dot{{\bf v}}_i\cdot{\bf n}=(\overline{{\bf
          v}}_i-{\bf v}_i)\cdot{\bf n}\leq 0$.}
    \label{fig:normal_vector}
  \end{figure}

  The bound of $d_V(t)$ implies that the spatial diameter of the flock, $d_X(t)$ grows at most linearly in time. Indeed, for agents ``$p$" and ``$q$" which realize the maximal distance, $d_X(t)=|\bx_p(t)-\bx_q(t)|$, we have
  \[
  \frac{d}{dt}d_X(t)\leq |{\bf v}_p(t)-{\bf v}_q(t)| \leq d_V(t),
  \] 
  and hence $d_X(t) \leq d_X(0)+ d_V(0)t$.
\end{remark}

Theorem \ref{thm:ine_dx_dv} and \ref{thm:ine_dx_dv_global} will be used to prove the flocking behavior of general systems of the type  (\ref{eqs:average}). The key point will be to make the judicious choice for the level $\theta=\theta(d_X(t))$, to enforce the convergence $d_V(t)\rightarrow 0$ through the inequalities (\ref{eqs:evo_d_XV}), (\ref{eqs:evo_d_XV_global}). In this context we are led to consider  dynamical inequalities of the form,
\begin{subequations}\label{eqs:HL}
  \begin{eqnarray}
    \frac{d}{dt}d_X(t) &\leq & d_V(t),\\
    \frac{d}{dt}d_V(t) & \leq & -\alpha\psi(d_X(t))d_V(t).
  \end{eqnarray}
\end{subequations}
The long time behavior of such systems is dictated by the properties of $\psi(\cdot)>0$.

\begin{lemma}\label{lem:HL}
  Consider the diameters $d_X(t),d_V(t)$ governed by the inequalities \eqref{eqs:HL}, where $\psi(\cdot)$ is a positive function such that,
  \begin{subequations}\label{eqs:psi}
    \begin{equation}\label{eq:psi}
      d_V(0) \leq \int^\infty_{d_X(0)} \psi(r)dr.
    \end{equation}
    Then the underlying dynamical system convergences to a flock in the sense that  \eqref{eq:cond_flock} holds,
    \[
    \sup_{t\geq 0}d_X(t) <+\infty \ \ \text{and} \ \ \lim_{t\rightarrow +\infty}d_V(t)=0.
    \]
    In particular, if $\psi(\cdot)$ has a diverging tail,
    \begin{equation}\label{eq:psi_diverge}
      \int^\infty \psi(r)dr = \infty,
    \end{equation}
  \end{subequations}
  then there is unconditional flocking.
\end{lemma}
\begin{proof}
  We apply the energy method introduced by Ha and Liu \cite{ha_simple_2009}. Consider the ``energy functional'', $\mathcal{E}={\mathcal E}(t)$,
  \begin{equation}
    \label{eq:phi}
    \mathcal{E}(d_X,d_V)(t):= d_V(t) + \alpha  \int_0^{d_X(t)} \psi(s)\,ds.
  \end{equation}
  The energy $\mathcal{E}$ is decreasing along the trajectory $(d_X,d_V)$, 
  \begin{eqnarray*}
    \frac{d}{dt}\, \mathcal{E}(d_X,d_V) &=& \dot{d_V} + \alpha \psi(d_X) \dot{d_X} 
    \leq  -\alpha\psi(d_X) d_V +  \alpha\psi(d_X) d_V \, = 0,
  \end{eqnarray*}
  and we deduce that,
  \begin{equation}
    \label{eq:super_inequality}
    d_V(t) -d_V(0) \leq  -\alpha\int_{d_X(0)}^{d_X(t)} \psi(s)\,ds.
  \end{equation}
  By our assumption (\ref{eq:psi}), there exists $d_*>0$ (independent of $t$), such that,
  \begin{equation}\label{eq:key}
    d_V(0) = \alpha\int_{d_X(0)}^{d_*} \psi(s)\,ds,
  \end{equation}
  and  the inequality (\ref{eq:super_inequality}) now reads,
  \begin{equation*}
    d_V(t) \leq \alpha\int_{d_X(0)}^{d_*} \psi(s)\,ds - \alpha\int_{d_X(0)}^{d_X(t)} \psi(s)\,ds = \alpha\int_{d_X(t)}^{d_*} \psi(s)\,ds.
  \end{equation*}
  Since $d_V(t)\geq0$, we conclude that we have a flock with a  \emph{uniformly  bounded} diameter,
  \begin{equation}
    \label{eq:upper_bound_x_M}
    d_X(t) \leq d_* \quad \text{ for all } t\geq0,
  \end{equation}
  thus improving the linear growth noted in remark \ref{rem:v_is_bounded}.  The uniform bound on $d_X(t)$ in (\ref{eq:upper_bound_x_M}) implies that the velocity phase space of this flock shrinks as the diameter $d_V(t)$ converges to zero. Indeed, the inequality (\ref{eq:evo_d_V_theta1}) yields,
  \begin{eqnarray*}
    \frac{d}{dt} d_V(t) & \leq & -\alpha \psi_* \cdot d_V(t), \qquad \psi_*:=\min_{0\leq r \leq d_*}\psi(r)>0,
  \end{eqnarray*}
  and  Gronwall's inequality proves that $d_V(t)$ converges exponentially fast to zero.
\end{proof}

\section{Flocking for the proposed model}\label{sec:new-model-flock}
\setcounter{equation}{0}
In this section we prove that the model (\ref{eq:eitan}) converges to a flock under the assumption that the pairwise influence, $\phi(|\bx_j-\bx_i|)$, decays slowly enough so that $\phi(\cdot)$, has a non square-integrable tail, (\ref{eq:phi2_not_integrable}), $\int^\infty \phi^2(r)dr=\infty$. In section \ref{sec:leader}, we show that the same result carries over the dynamics of strongly asymmetric models with leader(s). 
We will conclude, in section \ref{subsec:C-S-flock}, by revisiting the flocking behavior of the C-S model.

\subsection{Flocking of the proposed model}\label{subsec:new-model-flock}

\begin{theorem}\label{thm:flock_particle}
  Consider the  model for self-organized dynamics \eqref{eq:eitan} and assume that  its influence function $\phi$ satisfies,
  \begin{subequations}\label{eqs:non_square_int}
    \begin{equation}\label{eq:non_square_int_local}
      d_V(0) \leq \int_{d_X(0)}^\infty \phi^2(r)dr.
    \end{equation}
    Then, its solution, $\{({\bf x}_i(t),{\bf v}_i(t))\}_{i}$, converges to a flock in the sense that \eqref{eq:cond_flock} holds. In particular, there is unconditional flocking if $\phi^2$ has a diverging tail,
    \begin{equation}
      \label{eq:non_square_int}
      \int^\infty \phi^2(r)\,dr = +\infty.
    \end{equation}
  \end{subequations}
\end{theorem}
\begin{proof}
  Since   $\phi(d_X)\leq \phi_{ij} \leq 1$, the alignment coefficients $a_{ij}$ in (\ref{eqs:eitan_average}) are lower-bounded by
  \begin{displaymath}
    a_{ij} = \frac{\phi(|{\bf x}_j- {\bf x}_i|)}{\sum_k \phi(|{\bf x}_k- {\bf x}_i|)} \geq \frac{\phi(d_X)}{N}.
  \end{displaymath}
  We now set  $\theta$ to be this lower-bound of the $a_{ij}$\!'s,
  \[
  \theta(t):=\frac{\phi(d_X(t))}{N},
  \]
  so that the global active set at that level, $\Lambda(\theta(t))$, include \emph{all} agents. Thus, as noted already in  (\ref{eq:theta_1}), $\lambda(\theta)=N$, and the global version of our main theorem  \ref{thm:ine_dx_dv_global} yields,
  \begin{eqnarray*}
    \frac{d}{dt} d_X(t) & \leq & d_V(t) \\
    \frac{d}{dt} d_V(t) & \leq & -\alpha\,\phi^2(d_X(t)) \,d_V(t).
  \end{eqnarray*}
  The result follows from lemma \ref{lem:HL} with $\psi(r)=\phi^2(r)$.
\end{proof}

\begin{remark}\label{rem:differences}
  Theorem  \ref{thm:flock_particle} tells us that the model (\ref{eq:eitan}) admits an asymptotic flocking velocity, ${\bf v}_\infty$ 
  \[
  \lim {\bf v}_i(t) = {\bf v}_\infty.
  \]
  In contrast to the C-S model, however, our model does not seem to posses any invariants which will enable to relate ${\bf v}_\infty$ to the initial condition, beyond the fact noted in remark \ref{rem:v_is_bounded}, that ${\bf v}_\infty$ belongs to the convex hull $\Omega(0)$. We can therefore talk about the \emph{emergence} in the new model, in the sense that the asymptotic velocity of its flock, ${\bf v}_\infty$, is encoded in the dynamics of the system and not just as an invariant of its initial configuration.  Whether ${\bf v}_\infty$ can be computed from the initial configuration remains an open question.
\end{remark}

\subsection{Flocking with a leader}\label{sec:leader}

In this section, we discuss the dynamical systems  with (one or more) leaders.
\begin{definition}\label{defi:leader}
  Consider the dynamical system \eqref{eqs:eitan_average}. An agent ``$p$" is a leader if there exists $\beta>0$, independent of $N$, such that:
  \begin{equation}\label{eq:leader}
    a_{ip}(t) \geq \beta \phi(|\bx_p-\bx_i|), \quad \text{for every } i.
  \end{equation}
\end{definition}
In other words, an agent ``$p$" is viewed as a leader if its influence on aligning all other agents ``$i$", is decreasing with distance, but otherwise, is \emph{independent} of the number of agents, $N$. We illustrate this definition, see figure \ref{fig:sheep_herder}, with the following dynamical system: a leader ``$p$" moves with a constant velocity and influences the rest of the agents with a non-vanishing amplitude $0<\beta<1$,
\begin{subequations}\label{eqs:leader}
  \begin{equation}\label{eq:ex_leader}
    \frac{d{\bf x}_i}{dt} = {\bf v}_i, \qquad \frac{d{\bf v}_i}{dt} = \alpha\sum_{j\ne i}a_{ij}({\bf v}_j-{\bf v}_i),
  \end{equation}
  where
  \begin{equation}\label{eq:leaderb}
    a_{pj} = 0, \qquad a_{ij \big| i\ne p}=\left\{\begin{array}{ll} 
        \displaystyle \beta\phi(|\bx_p-\bx_i|), & j=p,\\ \\
        \displaystyle \frac{1-\beta}{N}\phi(|\bx_j-\bx_i|) & j\ne p.
      \end{array}\right.
  \end{equation} 
\end{subequations}


\begin{figure}[ht]
  \centering
  \includegraphics[scale=.35]{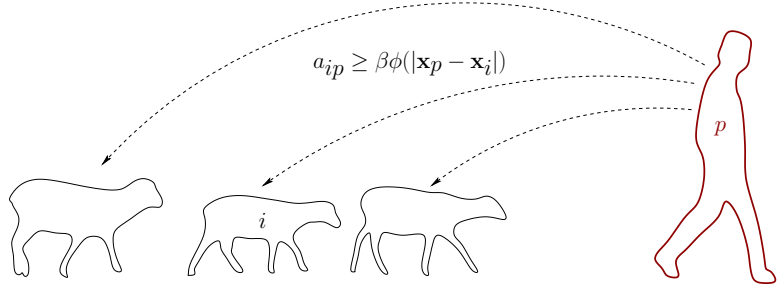}
  \caption{The agent $p$ (herder) is a leader in the sense of definition (\ref{defi:leader}). He influences every other agents (sheep) more than a certain quantity \mbox{$\beta\phi(|\bx_p-\bx_i|)$}.}
  \label{fig:sheep_herder}
\end{figure}

We note that there could be one or more leaders. The presence  of  leader(s) in the dynamical system (\ref{eqs:eitan_average}) is of course typical to asymmetric systems. We use the approach outlined above to prove that the existence of one (or more) leaders, enforces flocking.

\begin{theorem}
  Let $\{{\bf x}_i(t),{\bf v}_i(t)\}$ the solution of the dynamical system \eqref{eqs:eitan_average} and assume it has one or more leaders in the sense that \eqref{eq:leader} holds.  Then $\{{\bf x}_i(t),{\bf v}_i(t)\}$ admits a conditional and respectively, unconditional flocking provided \eqref{eq:non_square_int_local} and respectively, \eqref{eq:non_square_int} hold.
\end{theorem}
\begin{proof}
  We set $\theta=\beta\phi(d_X(t))$. Then the leader ``$p$" belongs to \emph{all} active sets, $\Lambda_i(\theta)$, and in particular, ``$p$" belongs to the global active set $\Lambda(\theta)$. Thus,  $\lambda(\theta)\geq1$. The inequalities (\ref{eqs:evo_d_XV}) yield
  \begin{eqnarray*}
    \frac{d}{dt} d_X(t) & \leq & d_V(t) \\
    \frac{d}{dt} d_V(t) & \leq & - \alpha \beta^2 \phi^2(d_X(t))d_V(t).
  \end{eqnarray*}
  We now apply lemma \ref{lem:HL} with $\psi(r)=\phi^2(r)$ to conclude.
\end{proof}
\begin{remark}
  If the leader $p$ is not influenced by the other agents, then one deduces that the asymptotic velocity of the flock ${\bf v}_\infty$ will be the velocity of the leader ${\bf v}_p$. But we emphasize that in the general case of having more than one leader the asymptotic velocity of the flock emerges through the dynamics of \eqref{eqs:eitan_average}, and as with the model \eqref{eq:eitan}, it may not be encoded solely in the initial configuration.
\end{remark}

\subsection{Flocking of the C-S model revisited}\label{subsec:C-S-flock}

We close this section by showing how the flocking behavior of the C-S model (\ref{eqs:CS}) can be studied using the framework outlined above.  By our assumption, the scaling of the influence function $\phi(\cdot)\leq 1$, we have
\[
\frac{1}{N}\sum_{j\neq i}\phi(|\bx_i-\bx_j|) \leq 1.
\]
Hence, we can recast the C-S model (\ref{eq:CS}) in the form (\ref{eqs:average})
\begin{equation}\label{eq:CS_inf}
  \frac{d{\bf v}_i}{dt} = \alpha 
  \sum_{j\ne i}^N a_{ij}({\bf v}_j-{\bf v}_i), \qquad 
  a_{ij}=\left\{\begin{array}{ll} 
      \displaystyle \frac{1}{N}\phi(|\bx_i-\bx_j|), & j\neq i,\\ \\
      \displaystyle 1-\frac{1}{N}\sum_{j\neq i} \phi(|\bx_i-\bx_j|), & j=i.
    \end{array} \right.
\end{equation}
In this case, $a_{ij}\geq \phi(d_X(t))/N$ for $j\ne i$. Moreover, the same lower-bound applies for $j=i$, because of the normalization $\phi\leq 1$: 
\[
a_{ii}=1-\frac{1}{N}\sum_{j\neq i}\phi(|\bx_i-\bx_j|) \geq 1-\frac{N-1}{N} \geq \frac{\phi(d_X(t))}{N}.
\]
Therefore, if we now set  $\theta$ to be this lower-bound of the $a_{ij}$\!'s,
\[
\theta(t):=\frac{\phi(d_X(t))}{N},
\]
then $\Lambda_{p}(\theta(t))$, and consequently, $\Lambda(\theta)$, include \emph{all} agents, $\lambda(\theta)=N$, consult (\ref{eq:theta_1}). Theorem \ref{thm:ine_dx_dv_global} yields,
\begin{subequations}
  \begin{eqnarray}
    \frac{d}{dt} d_X(t) & \leq & d_V(t) \label{eq:evo_d_X_theta1}\\
    \frac{d}{dt} d_V(t) & \leq & -\alpha\,\phi^2(d_X(t)) \,d_V(t). \label{eq:evo_d_V_theta1}
  \end{eqnarray}
\end{subequations}
Now, apply lemma \ref{lem:HL} with $\psi(r)=\phi^2(r)$ to conclude the following.

\begin{corollary}
  Consider the C-S model \eqref{eqs:CS} with an influence function, $\phi$, that has a non square-integrable tail, \eqref{eq:phi2_not_integrable}.  Then the C-S solution, $\{({\bf x}_i(t),{\bf v}_i(t))\}_{i}$, converges, unconditionally, to a flock in the sense that \eqref{eq:cond_flock} holds. In particular, since the total momentum is conserved, \eqref{eq:mom},
  \[
  {\bf v}_i(t) \stackrel{t \rightarrow \infty}{\longrightarrow} \overline{{\bf v}}\; ,\quad 
  \overline{{\bf v}}:= \frac{1}{N} \sum_i {\bf v}_i(0).
  \]
\end{corollary}

Comparing the quadratic divergence (\ref{eq:non_square_int}) vs. the sharp condition for C-S flocking, (\ref{eq:phi2_not_integrable}), we observe that the unconditional C-S flocking we derive in this case requires a more stringent condition of the influence function. This is due to the fact that the proposed approach for analyzing flocking is more versatile, being independent whether the underlying model is symmetric or not.

\section{From particle to mesoscopic description}\label{sec:meso}
\setcounter{equation}{0}

We would like to study the model (\ref{eq:eitan}) when the number of particles $N$ becomes large. With this aim, it is more convenient to study the kinetic equation associated with the dynamical system (\ref{eq:eitan}). The purpose of the section is precisely to derive formally such equation.

We introduce the so-called empirical distribution \cite{spohn_large_1991} of particles
$f^N(t,{\bf x},{\bf v})$, 
\begin{equation}
  \label{eq:f_N}
  f^N(t,{\bf x},{\bf v}) := \frac{1}{N} \sum_{i=1}^N \delta_{{\bf x}_i(t)} \otimes \delta_{{\bf v}_i(t)},
\end{equation}
where $\delta_{\bf x}\otimes\delta_{\bf y}$ is the usual Dirac mass on the phase space $\mathbb{R}^d \times \mathbb{R}^d$. Integrating the empirical distribution $f^N$ in the velocity variable ${\bf v}$ gives the density distribution of particles $\rho^N(t,{\bf x})$ in space,
\begin{equation}
  \label{eq:rho_N}
  \rho^N(t,{\bf x}) = \frac{1}{N} \sum_{i=1}^N \delta_{{\bf x}_i(t)}({\bf x}).
\end{equation}
Using the distributions $f^N$ and $\rho^N$, the particle system (\ref{eq:eitan}) reads,
\begin{subequations}
  \label{eqs:disc_liouville}
  \begin{eqnarray}
    \frac{d{\bf x}_i}{dt} &=& {\bf v}_i, \\
    \frac{d{\bf v}_i}{dt} &=& \alpha \frac{\int_{{\bf y},{\bf w}} \phi(|{\bf y}\!-\!{\bf x}_i|)\,({\bf w}\!-\!{\bf
        v}_i) \, f^N({\bf y},{\bf w})\,d{\bf y} d{\bf w}}{\int_{\bf y} \phi(|{\bf y}\!-\!{\bf x}_i|)\,\rho^N({\bf
        y})\,d{\bf y}}.
  \end{eqnarray}
\end{subequations}
Therefore, we can easily check that the empirical distribution $f^N$ satisfies (weakly) the Liouville equation,
\begin{subequations}\label{eqs:liouville}
  \begin{equation}
    \label{eq:evo_f}
    \partial_t f + {\bf v}\cdot \nabla_{\bf x} f + \nabla_{{\bf v}}\cdot(F[f]\,f) = 0,
  \end{equation}
  where the vector field $F[f]$ and the total mass $\rho$ are given by,
  \begin{equation}
    \label{eq:F}
    F[f]({\bf x},{\bf v}) := \alpha \frac{\int_{{\bf y},{\bf w}} \phi(|{\bf y}\!-\!{\bf x}|)\,({\bf w}\!-\!{\bf v})f({\bf
        y},{\bf w})\,d{\bf y} d{\bf w}}{\int_{\bf y} \phi(|{\bf y}\!-\!{\bf x}|) \,\rho({\bf y})\,d{\bf
        y}} \;, \quad   \rho({\bf y}) = \int_{\bf w} f({\bf y},{\bf w})\,d{\bf w}.
  \end{equation}
\end{subequations}
To study the limit as the number of particles $N$ approaches infinity, we first assume that the initial condition $f_0^N({\bf x},{\bf v})$ converges to a smooth function $f_0({\bf x},{\bf v})$ as $N\rightarrow +\infty$. Then it is natural to expect that $f^N(t,{\bf x},{\bf v})$ convergences to the solution $f(t,{\bf x},{\bf v})$ of the kinetic equation,
\begin{equation}
  \label{eq:evo_f2}
  \left\{
    \begin{array}{l}
      \partial_t f + {\bf v}\cdot \nabla_{\bf x} f + \nabla_{\bf v}\cdot(F[f]\,f) = 0, \\
      f_{t=0} = f_0
    \end{array}
  \right.
\end{equation}
However, the passage from the discrete system (\ref{eqs:disc_liouville}) to the kinetic formulation (\ref{eqs:liouville}) is more delicate than in the argument for the C-S model \cite{ha_simple_2009,ha_particle_2008}: here, the vector field $F[f]$ may not posses enough Lipschitz regularity due to the normalizing factor at the denominator of (\ref{eq:F}).  But since this question does not play a central in the scope of this paper, we leave the study of existence and uniqueness of solution of the kinetic equation (\ref{eqs:liouville}) for a future work, and we turn our focus to the \emph{hydrodynamic} model.

\section{Hydrodynamics of the proposed model and its flocking behavior}\label{sec:hydro}
\setcounter{equation}{0}

Having the kinetic description associated with the particle dynamics (\ref{eq:eitan}), we can derive the macroscopic limit of the dynamics \cite{degond_large_2008,degond_continuum_2008,ha_particle_2008}. We also extend our method developed in section \ref{sec:tools} to prove the flocking behavior of the model in the macroscopic case. To this end we extend the notion of active sets from the discrete setup the continuum, and the corresponding key algebraic lemma \ref{lem:fundamental_lemma} for skew-symmetric integral operators.

\subsection{Macroscopic system}\label{sec:macro}
To derive the macroscopic model of the particle system (\ref{eq:eitan}), we just integrate the kinetic equation (\ref{eq:evo_f}) in the phase space. With this aim, we first define the macroscopic velocity ${\bf u}$ and the pressure term ${\bf P}$,
\begin{displaymath}
  \rho(t,{\bf x}){\bf u}(t,{\bf x}) = \int_{\bf v} {\bf v} f(t,{\bf x},{\bf v})\,d{\bf v} \quad,
  \quad   {\bf P}(t,{\bf x})=\int_{\bf v} ({\bf v}-{\bf u})\otimes({\bf v}-{\bf u})f(t,{\bf x},{\bf v})\,d{\bf v},
\end{displaymath}
where $\rho$ is the spatial density defined previously (\ref{eq:F}). Then integrating the kinetic equation (\ref{eq:evo_f}) against the first moments $(1,{\bf v})$ yields the system (see also \cite{ha_particle_2008}),
\begin{subequations}\label{eqs:macro}
  \begin{eqnarray}
    \label{eq:macro_rho_tp}
    && \partial_t \rho + \nabla_{\bf x} \cdot (\rho{\bf u}) =0 \\
    \label{eq:macro_u_tp}
    && \partial_t (\rho {\bf u}) + \nabla_{\bf x}\cdot(\rho{\bf u} \otimes {\bf u} + {\bf P}) = {\bf S}({\bf u}),
  \end{eqnarray}
  where the source term ${\bf S}({\bf u})$ is given by, (recall the notation of (\ref{eq:hydro_con2}),  $\langle {\bf w}\rangle =\phi*({\bf w}\rho$)),
  \begin{equation}
    \label{eq:source_term}
    {\bf S}({\bf u})(\bx) = \alpha\,\frac{\ds \int_{\bf y} \phi(|{\bf y}\!-\!{\bf x}|)\rho({\bf x})\rho({\bf y})\big({\bf u}({\bf y})-{\bf u}({\bf
        x})\big)\,d{\bf y}}{\ds \int_{\bf y} \phi(|{\bf y}\!-\!{\bf x}|)\rho({\bf y})\,d{\bf y}} = \alpha \rho(\bx)\left(\frac{\langle {\bf u}\rangle(\bx)}{\langle 1\rangle(\bx)} - {\bf u}(\bx)\right).
  \end{equation}
\end{subequations}
The system (\ref{eqs:macro}) is not closed since the equation for $\rho{\bf u}$ (\ref{eq:macro_u_tp}) does depend on the third moment of $f$ which is encoded in the pressure term ${\bf P}$. In order to close the system, we neglect the pressure, setting ${\bf P}=0$ (in other words, we assume a monophase distribution, $ f(t,{\bf x},{\bf v}) = \rho(t,{\bf x})\, \delta_{{\bf u}(t,{\bf x})}({\bf v})$).  Under this assumption, (\ref{eqs:macro}) is reduced to the closed system (\ref{eqs:hydro_cons}),
\begin{subequations}\label{eqs:macro_closed}
  \begin{eqnarray}
    \label{eq:macro_rho}
    && \partial_t \rho + \nabla_{\bf x} \cdot (\rho{\bf u}) =0 \\
    \label{eq:macro_u}
    && \partial_t (\rho {\bf u}) + \nabla_{\bf x}\cdot(\rho{\bf u} \otimes {\bf u}) = {\bf S}({\bf u}).
  \end{eqnarray}
\end{subequations}

We want to study the flocking behavior of \emph{general} systems of the form 
(consult figure \ref{fig:active_set_macro}),
\begin{subequations}\label{eqs:macro_average}
  \begin{eqnarray}
    \label{eq:macro_rho_noncons}
    && \partial_t \rho + \nabla_{\bf x} \cdot (\rho{\bf u}) =0 \\
    \label{eq:macro_u_noncons}
    && \partial_t {\bf u} + ({\bf u}\cdot\nabla_{\bf x}){\bf u} = \alpha(\overline{{\bf u}}-{\bf u}).
  \end{eqnarray}
  The expression on the right reflects the tendency of agents with velocity ${\bf u}$ to relax to the \emph{local average velocity}, $\overline{{\bf u}}({\bf x})$, dictated by the influence function  $a({\bf x},{\bf y})$,
  \begin{equation}
    \label{eq:over_u}
    \overline{{\bf u}}({\bf x}) = \int_{\bf y} a({\bf x},{\bf y}) \rho({\bf y}) {\bf u}({\bf y})\,d{\bf y}, \quad \int_{{\bf y}} a({\bf x},{\bf y})\rho({\bf y})\,d{\bf y} = 1.
  \end{equation}
\end{subequations}
The class of equations (\ref{eqs:macro_average}) includes, in particular, the hydrodynamic description of our self-organized dynamics model, (\ref{eqs:macro_closed}), with
\begin{equation}
  \label{eq:a_macro}
  a({\bf x},{\bf y}) = \frac{\phi(|{\bf y}-{\bf x}|)}{\int_{{\bf y}}\phi(|{\bf y}-{\bf x}|)\rho({\bf y})\,d{\bf y}}.
\end{equation}
 
\begin{figure}[ht]
  \centering
  \includegraphics[scale=1]{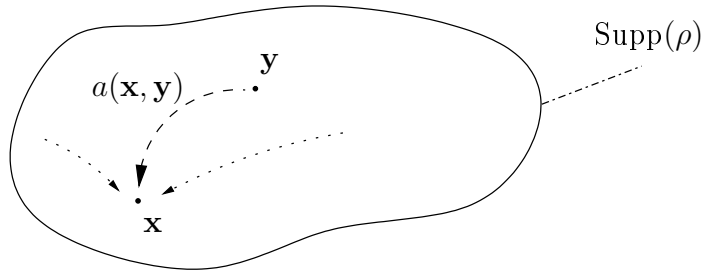}
  \caption{The quantity $a({\bf x},{\bf y})$ (\ref{eq:a_macro}) is the relative influence
    of the particles in ${\bf y}$ on the particles in ${\bf x}$.}
  \label{fig:active_set_macro}
\end{figure}

We begin with the definition of a flock in the macroscopic case.

\begin{definition}\label{defi:flock_macro}
  Let $\rho(t,{\bf x})>0$ and ${\bf u}(t,{\bf x}))$ be the density and velocity vector field which solve \eqref{eqs:macro_average}. Let $\text{Supp}(\rho)$ denotes the non-vacuum states,
  \begin{displaymath}
    \text{Supp}(\rho) := \overline{\{{\bf x}\in \mathbb{R}^d\ \big| / \ \rho({\bf x}) \neq 0\}},
  \end{displaymath}
  and consider the diameters, $d_X(t)$ and $d_V(t)$, of $\rho$ and, respectively, ${\bf u}$, 
  \begin{subequations}\label{eqs:d_XV_macro}
    \begin{eqnarray}
      \label{eq:d_X_macro}
      d_X(t) &:=& \sup \{|{\bf x}-{\bf y}| \,,\;\; {\bf x},{\bf y} \in \text{Supp}(\rho(t))\},\\
      \label{eq:d_V_macro}
      d_V(t) &:=& \sup \{|{\bf u}(t,{\bf x})-{\bf u}(t,{\bf y})| \, , \;\; {\bf x},{\bf y} \in \text{Supp}(\rho(t))\}.
    \end{eqnarray}
  \end{subequations}
  The solution $(\rho,{\bf u})$ converges to a flock if its diameters satisfy,
  \begin{equation}
    \label{eq:flock_macro}
    \sup_{t\geq0}\, d_X(t) < +\infty \qquad \text{and} \qquad \lim_{t\rightarrow+\infty} d_V(t) =0.
  \end{equation}
\end{definition}
Clearly, in order to have a flock, the initial density, $\rho_0$, needs to be compactly supported. Furthermore, we also impose that the initial velocity, ${\bf u}_0$, has a compact support, assuming:
\begin{equation}
  \label{eq:supp_bound}
  d_X(0) <+\infty \quad \text{ and } \quad d_V(0) <+\infty.
\end{equation}
In the following, we assume there exists a smooth solution $(\rho,{\bf u})$ of the system (\ref{eqs:macro_average}).

\medskip

\noindent \underline{{\bf Hypothesis}}: {\it Consider the system \eqref{eqs:macro_average} subject to compactly supported initial data, $(\rho_0,\,u_0)$,  \eqref{eq:supp_bound}. We
  assume that  it admits a unique smooth solution $(\rho(t),{\bf u}(t))$ for all $t\geq 0$.}

\subsection{Active sets at the macroscopic scale}

To prove that the solution $(\rho,{\bf u})$ converges to a flock, we need to show that the convex hull in velocity space,
\[
\Omega(t):= Conv\{{\bf u}(t,\bx) \ | \ \bx \in Supp(\rho(t,\cdot))\},
\]
shrinks to a single point, as its diameter, $d_V(t)$, converges to zero. To this end, we employ the notion of \emph{active sets} which is extended to the present context of  macroscopic framework. We begin by revisiting our definition  of active set using the influence function $a({\bf x},{\bf  y})$ (\ref{eq:a_macro}).

\begin{definition}
  Fix $\theta>0$. For every ${\bf x}$ in the support of $\rho$, we define the active set, $\Lambda_{\bf x}(\theta)$, as
  \begin{equation}
    \label{eq:active_set_macro}
    \Lambda_{\bf x}(\theta) = \{ {\bf y} \in \text{Supp}(\rho) \ \big| \ a({\bf x},{\bf y}) \geq \theta \}.
  \end{equation}
  The global active set $\Lambda(\theta)$ is the intersection of all the active set $\Lambda_{\bf x}(\theta)$:
  \begin{equation}
    \label{eq:active_set_global_macro}
    \Lambda(\theta) = \!\!\!\bigcap_{{\bf x} \in \text{Supp}(\rho)} \!\!\! \Lambda_{\bf x}(\theta) \;=\; \left\{ {\bf y} \in \text{Supp}(\rho) \ \big| \
      a({\bf x},{\bf y}) \geq \theta  \; \text{ for all } {\bf x} \text{ in }
      \text{Supp}(\rho)\right\}.
  \end{equation}
\end{definition}
As before, we let $\lambda_{\mathcal I}(\theta)$ denote the density of agents in the corresponding active set; thus
\begin{equation}
  \label{eq:theta_macro}
  \lambda_{\bf x}(\theta) := \int_{\Lambda_{\bf x}(\theta)} \rho({\bf y})\,d{\bf y}, \qquad  \lambda(\theta) = \int_{\Lambda(\theta)} \rho({\bf y})\,d{\bf y}.
\end{equation}

We would like to extend the key lemma \ref{lem:fundamental_lemma} from the discrete case of agents to the  macroscopic case of the continuum. This is formulated in terms of  the maximal action of integral operators which involve  antisymmetric kernels, $k({\bf x},{\bf y})$. 
\begin{lemma}
  \label{lem:fundamental_lemma_macro}
  Let $\rho \in L^1(\mathbb{R}^d)$ be a positive function and let $k$ be a bounded antisymmetric kernel, $|k({\bf x},{\bf y})| \leq M$ and \mbox{$k({\bf x},{\bf y})=-k({\bf y},{\bf x})$}. Fix two positive functions, $u$ and $w$ in $L_\rho^1$ with a total mass $\overline{U}$ and $\overline{W}$,
  \begin{equation}
    \overline{U} = \int_{\bf x} u({\bf x}) \rho({\bf x})\,d{\bf x}  \quad \text{and} \quad
    \overline{W} = \int_{\bf x} w({\bf x}) \rho({\bf x})\,d{\bf x}.
  \end{equation}
  Then, for every positive number $\theta$, we have:
  \begin{equation}
    \label{eq:S_ine_macro}
    \left| \int_{{\bf x},{\bf y}} k({\bf x},{\bf y}) u({\bf x}) w({\bf y})\,\rho({\bf x})\rho({\bf y})\,d{\bf x} d{\bf y} \right| \leq
    M \overline{U}\,\overline{W}\,\big(1-\lambda^2(\theta) \theta^2\big).
  \end{equation}
  Here, $\lambda(\theta)$ is the density of active agents at level $\theta$ for $u$ and $w$,
  \begin{displaymath}
    \lambda(\theta) = \int_{\Lambda_{u,w}(\theta)}  \rho({\bf x})\,d{\bf x},\qquad
    \Lambda_{u,w}(\theta):= \{ {\bf x} \in \text{Supp}(\rho)\ \big| \ u({\bf x}) \geq \theta\,\overline{U} \; \text{ and }
    \; w({\bf x}) \geq \theta\,\overline{W} \}.
  \end{displaymath}
\end{lemma}
\begin{proof}
  To simplify, we denote $\mathcal{S}:=\int_{{\bf x},{\bf y}} k({\bf x},{\bf y})\,u({\bf x})w({\bf y})\,\rho({\bf x})\rho({\bf y})\,d{\bf x}d{\bf y}$. The anti-symmetry of $k$ enables us to rewrite,
  \begin{equation*}
    \mathcal{S} = \frac{1}{2} \int_{{\bf x},{\bf y}}
    k({\bf x},{\bf y})\Big[u({\bf x})w({\bf y}) -u({\bf y})w({\bf x})\Big]\rho({\bf
      x})\rho({\bf y})\,d{\bf x}d{\bf y}.
  \end{equation*}
  The bound on $k$ and the identity $|a-b| \equiv a+b -2 \min(a,b)$ yields,
  \begin{displaymath}
    \left|\mathcal{S}\right| \leq   \frac{1}{2} \int_{{\bf x},{\bf y}} \!\!\!\! M \Big[u({\bf x})w({\bf y}) + u({\bf y})w({\bf x}) -\;2\min\big(u({\bf x})w({\bf
      y}),u({\bf y})w({\bf x})\big)\Big]\rho({\bf x})\rho({\bf y})\,d{\bf x}d{\bf y}.
  \end{displaymath}
  Using the notations, we obtain,
  \begin{displaymath}
    \left|\mathcal{S}\right| \leq  M \overline{U}\,\overline{W} - M \int_{{\bf x},{\bf y}} \min\big(u({\bf
      x})w({\bf y}),u({\bf y})w({\bf x})\big)\rho({\bf x})\rho({\bf y})\,d{\bf x}d{\bf y}.
  \end{displaymath}
  We now restrict the domain of integration on the right hand side to $(\bx,\by)\in \Lambda_\theta\times\Lambda_\theta$, where the lower-bounds of $u$ and $w$ yield,
  \begin{eqnarray*}
    \left|\mathcal{S}\right| & \leq & M \overline{U}\,\overline{W} - M \int_{\Lambda_\theta\times\Lambda_\theta} \theta \overline{U}\,\theta \overline{W} \,\rho({\bf x})\rho({\bf y})\,d{\bf x}d{\bf y}
    = M \overline{U}\,\overline{W} - M \theta^2 \overline{U}\,\overline{W}\lambda^2(\theta),
  \end{eqnarray*}
  and (\ref{eq:S_ine_macro}) follows.
\end{proof}

\subsection{Decay of the diameters}\label{sec:decay_macro_diam}

The diameters $(d_X,d_V)$ also satisfy the same inequality at the macroscopic level. We only need to adapt the proof using the characteristics of the system (\ref{eqs:macro_average}).
\begin{proposition}
  \label{ppo:ine_dx_dv_macro}
  Let $(\rho,{\bf u})$ the solution of the dynamical system \eqref{eqs:macro_average}. Fix an arbitrary $\theta$ and let $\lambda(\theta)$ be the density of agents on the corresponding global active set $\Lambda(\theta)$ associated with this system, \eqref{eq:theta_macro}. Then, the diameters $d_X(t)$ and $d_V(t)$ in \eqref{eqs:d_XV_macro} satisfy,
  \begin{subequations}\label{eqs:evo_d_XV_macro}  
    \begin{eqnarray}
      \label{eq:evo_d_X_macro}
      \frac{d}{dt} d_X(t) & \leq & d_V(t) \\
      \label{eq:evo_d_V_macro}
      \frac{d}{dt} d_V(t) & \leq & -\alpha\,\lambda^2(\theta)\,\theta^2(t) \,d_V(t).
    \end{eqnarray}
  \end{subequations}
\end{proposition}
\begin{proof}
  We fix our attention on two characteristics $\dot{X}(t)={\bf u}(t,X)$ and $\dot{Y}(t)={\bf u}(t,Y)$, subject to initial conditions, $X(0)=\bx$ and $Y(0)=\by$ for two points $({\bf x},{\bf y})$ in the support of $\rho(0)$. Their relative distance satisfy:
  \begin{displaymath}
    \frac{d}{dt} |Y-X|^2 = 2 \langle \,Y \!-\! X\,,\,{\bf u}(Y) - {\bf u}(X)\,\rangle  \leq  2 d_X \,d_V.
  \end{displaymath}
  Since this inequality is true for every characteristics, (\ref{eq:evo_d_X_macro}) follows.

  We turn to study the relative distance in velocity phase space: using (\ref{eq:over_u}) we find,
  \begin{displaymath}
    \frac{d}{dt}\,|{\bf u}(Y)-{\bf u}(X)|^2 = 2\alpha\,\langle {\bf u}(Y)-{\bf u}(X)\,,\, \overline{{\bf u}}(Y) -
    \overline{{\bf u}}(X) - \big({\bf u}(Y)-{\bf u}(X)\big)\rangle,
  \end{displaymath}
  and hence,
  \begin{equation}
    \label{eq:evo_u_p_u_q}
    \frac{d}{dt}\,|{\bf u}(Y)-{\bf u}(X)|^2 \leq 2\alpha\,d_V \Big(|\overline{{\bf u}}(Y)
    - \overline{{\bf u}}(X)| - |{\bf u}(Y)-{\bf u}(X)|\Big).
  \end{equation}
  Using the fact that $a(X,\cdot)$ has a unit $\rho$-mass,  the difference of averages  $\overline{{\bf u}}(Y) - \overline{{\bf u}}(X)$ can be expressed as 
  \begin{eqnarray*}
    \overline{{\bf u}}(Y) - \overline{{\bf u}}(X) &=& \int_{{\bf w}} a(Y,{\bf w}) \rho({\bf w}){\bf u}({\bf w})\,d{\bf w} - \int_{{\bf z}}
    a(X,{\bf z}) \rho({\bf z}){\bf u}({\bf z})\,d{\bf z} \nb \\
    &=& \int_{{\bf w},{\bf z}} a(Y,{\bf w}) \rho({\bf w}) {\bf u}({\bf w})\, a(X,{\bf z}) \rho({\bf z})\,d{\bf w}
    d{\bf z} \\
    &&  - \;\int_{{\bf w},{\bf z}} a(X,{\bf z}) \rho({\bf z}){\bf u}({\bf z}) \,a(Y,{\bf w}) \rho({\bf w}) \,d{\bf w} d{\bf z} \\
    &=& \int_{{\bf w},{\bf z}} a(Y,{\bf w}) a(X,{\bf z}) \big[{\bf u}({\bf w})-{\bf
      u}({\bf z})\big]\rho({\bf w})\rho({\bf z})\,d{\bf w}d{\bf z}.
  \end{eqnarray*}
  We now appeal to the maximal action lemma, \ref{lem:fundamental_lemma_macro}, with anti-symmetric kernel, $k({\bf w},{\bf z}) = {\bf u}({\bf w})-{\bf u}({\bf z})$, and the positive functions $u({\bf w}) = a(Y,{\bf w})$ and $w({\bf z}) = a(X,{\bf z})$: since
  \begin{displaymath}
    |{\bf u}({\bf y})-{\bf u}({\bf z})| \leq d_V \;\;, \;\; \overline{U} = \int_{\bf y} a(Y,{\bf
      y})\rho({\bf y})\,d{\bf y} = 1 \; \text{ and } \overline{W} = \int_{\bf y} a(X,{\bf y})\rho({\bf y})\,d{\bf y} =1,
  \end{displaymath}
  we obtain,
  \begin{displaymath}
    |\overline{{\bf u}}(Y) - \overline{{\bf u}}(X)| \leq d_V\big(1-\lambda^2(\theta) \theta^2\big),
  \end{displaymath}
  Inserted into (\ref{eq:evo_u_p_u_q}), we end up with 
  \[
  \frac{d}{dt}\,|{\bf u}(Y)-{\bf u}(X)|^2 \leq 2\alpha\,d_V \Big(d_V(1-\lambda^2(\theta)\theta^2) - |{\bf u}(Y)-{\bf u}(X)|\Big).
  \]
  Finally, since the support of $\rho$ is compact, we can take the two characteristics $Y(t)$ and $X(t)$ such that at time $t$ we have: $d_V(t) = |u(Y)-u(X)|$, and the last inequality yields (\ref{eq:evo_d_V_macro}),
  \begin{displaymath}
    \frac{d}{dt}\,d_V^2(t) \leq 2\alpha\,d_V(t) \Big(d_V(t)\big(1-\lambda^2(\theta)\theta^2\big) -d_V(t)\Big)= -2\alpha\lambda^2(\theta)\theta^2 d_V^2(t).
  \end{displaymath}
\end{proof}

\subsection{Flocking in the hydrodynamic limit}\label{sec:hydro_flock}
Since the diameters $d_X$ and $d_V$ satisfy the same system of inequalities at the macroscopic level, (\ref{eqs:evo_d_XV_macro}), as in the particle level, (\ref{eqs:evo_d_XV}), we immediately deduce that theorem \ref{thm:flock_particle} is still valid for the macroscopic system (\ref{eqs:macro_average}).
\begin{theorem}
  Let $(\rho,{\bf u})$ the solution of the system \eqref{eqs:macro_average}. If the influence kernel, $\phi$, decays sufficiently slow,
  \begin{equation}
    \label{eq:nonint_macro}
    \int_0^\infty \phi^2(r)\,dr = +\infty,
  \end{equation}
  then $(\rho,{\bf u})$ converges to a flock in the sense of definition \ref{defi:flock_macro}.
\end{theorem}
\begin{proof}
  For every ${\bf x}$ and ${\bf y}$ in the support of  $\rho(t,\cdot)$, we have,
  \begin{displaymath}
    a({\bf x},{\bf y}) = \frac{\phi(|{\bf y}-{\bf x}|)}{\int_{{\bf y}}\phi(|{\bf y}-{\bf x}|)\rho({\bf
        y})\,d{\bf y}} \geq \frac{\phi(d_X)}{\int_{{\bf y}}\phi(0)\rho({\bf
        y})\,d{\bf y}} = \frac{\phi(d_X)}{\overline{\rho}}, \quad \overline{\rho}:=\int_{\by} \rho(t,\by)d\by \equiv \int_{\by} \rho_0(\by)d\by.
  \end{displaymath}
  Thus, if we take $\theta(t) = \phi(d_X(t))/\overline{\rho}$, every point ${\bf y}$ in the support of $\rho(t,\cdot)$ belongs to the global active set $\Lambda(\theta)$. Therefore, for this choice of $\theta$, we have,
  \begin{displaymath}
    \lambda(\theta) = \int_{\Lambda(\theta)} \rho({\bf x})\,d{\bf x} = \int_{Supp(\rho)} \rho({\bf x})\,d{\bf x} = \overline{\rho}.
  \end{displaymath}
  We deduce that,
  \begin{eqnarray*}
    \frac{d}{dt} d_X(t) & \leq & d_V(t) \\
    \frac{d}{dt} d_V(t) & \leq & -\alpha\,\phi^2(d_X(t)) \,d_V(t).
  \end{eqnarray*}
  To conclude, we apply lemma \ref{lem:HL} with $\psi(r)=\phi^2(r)$.
\end{proof}

\section{Conclusion}\label{sec:conc}
There is a large number of models for self-organized dynamics \cite{aoki_simulation_1982,ballerini_interaction_2008,buhl_disorder_2006,camazine_self-organization_2001,couzin_collective_2002,gregoire_onset_2004,grimm_individual-based_2005,hemelrijk_self-organized_2008,huth_simulation_1992,huth_simulation_1994,reynolds_flocks_1987,parrish_self-organized_2002,viscido_individual_2004,vicsek_novel_1995,youseff_parallel_2008}. 
In this paper we studied a general class of models for self-organized dynamics which take the form (\ref{eqs:eitan_average}),
\[
\frac{d{\bf x}_i}{dt} = {\bf v}_i, \qquad \frac{d{\bf v}_i}{dt} = \alpha 
\sum_{j\ne i}^N a_{ij}({\bf v}_j-{\bf v}_i), \qquad a_{ij} \geq 0, \ \sum_j a_{ij}=1.
\]
We  focused our attention on the popular Cucker-Smale model, \cite{cucker_emergent_2007,cucker_mathematics_2007}. Its dynamics is governed by symmetric interactions, $a_{ij}=\phi_{ij}/N$, involving a decreasing influence function $\phi_{ij}:=\phi(|x_i-x_j|)$. Here we  introduced an improved model where the interactions between agents is governed by the relative distances, $a_{ij}= \phi_{ij}/\sum_k \phi_{ik}$, which are no longer symmetric. To study the flocking behavior of such asymmetric dynamics, we based our analysis on the amount of influence that agents exert on each other. Using the so-called active sets, we were able to find explicit criteria for the unconditional emergence of a flock. In particular, we  derived a sufficient condition for flocking of our proposed model: flocking occurs independent of the initial configuration, when the interaction function $\phi$ decays sufficiently slow so that its tail is not square integrable, (\ref{eq:phi2_not_integrable}). Similar results holds for models with one or more leaders. This is only slightly more restrictive than the characterization of unconditional flocking in the symmetric case, which requires a non-integrable tail of $\phi$, (\ref{eq:phi_not_integrable}).

In either case, these requirements exclude compactly supported $\phi$'s: unconditional flocking is still restricted by the requirement that each agent is influenced by everyone else. A more realistic requirement is to assume that $\phi$ is rapidly decaying or that the influence function is cut-off at a finite distance. Here, there are two possible scenarios: (i) \emph{conditional} flocking, namely, flocking occurs if $d_V(0)$ and $d_X(0)$ are not too large relative to the rapid decay of $\phi^2$, $d_V(0) \leq \int^\infty_{d_X(0)} \phi^2(r)dr$; (ii) a remaining main challenge is to analyze the emergence of flocking in the \emph{general} case of compactly supported interaction function $\phi$. Clearly, this will have to take into account the \emph{connectivity} of the underlying graph ${\mathcal G}$, (\ref{eq:graph}).  We expect that the notion of active sets will be particularly relevant in this context of compactly supported $\phi$'s. The main difficulty is \emph{counting} the number of ``connected'' agents in the corresponding active sets.  As a prototypical example for the difficulties which arise with both --- asymmetric models and compactly supported interactions, we consider self-organized dynamics which involves \emph{vision}, where each agent has a \emph{cone of vision},
\begin{equation}
  \label{eq:CS_vision}
  \frac{d{\bf x}_i}{dt} = {\bf v}_i, \qquad \frac{d{\bf v}_i}{dt} = \alpha
  \sum_{j=1}^N\, \kappa(\omega_i,{\bf x}_j-{\bf x}_i)\,\phi_{ij} ({\bf v}_j-{\bf v}_i).
\end{equation}
Here, $\kappa(\omega_i,\bx_j-\bx_i)$ determines whether the agent ``$i$", heading in direction ${\bf \omega}_i:={\bf v}_i/|{\bf v}_i|$, ``sees'' the agent ``$j$":
\begin{displaymath}
  \kappa({\bf \omega}_i,{\bf x}_j-{\bf x}_i) = \left\{
    \begin{array}{ll}
      1, & \displaystyle \text{if }{\bf \omega}_i\cdot\frac{{\bf x}_j-{\bf x}_i}{|{\bf
          x}_j-{\bf x}_i|} \geq \gamma >-1, \\ \\
      0, & \text{otherwise},
    \end{array}
  \right.
\end{displaymath}
with $\gamma$ being the radius of the cone of vision $($see figure \ref{fig:vision}$)$. The $\phi_{ij}$\!'s  determine the pairwise alignment within the cone of vision, and can be modeled  either after C-S \eqref{eqs:CS},
\[
\phi_{ij}=\frac{1}{N_i}\phi(|\bx_j-\bx_i|), \qquad N_i:=\#\{j \ | \ \kappa(\omega_i,\bx_j-\bx_i)=1\},
\]
or after our proposed model for alignment, \eqref{eq:eitan}
\[
\phi_{ij}=\frac{\phi(|\bx_j-\bx_i|)}{\sum_{k=1}^{N_i} \phi(|\bx_k-\bx_i|)}.
\] 
In either case, the resulting model \eqref{eq:CS_vision} reads, 
\[
\frac{d{\bf v}_i}{dt} = \alpha
\sum_{j=1}^N\, a_{ij}  ({\bf v}_j-{\bf v}_i), \qquad a_{ij}=\kappa({\bf \omega}_i,{\bf x}_j-{\bf x}_i)\,\phi_{ij},
\]
and it lacks symmetry, $a_{ij}\neq a_{ji}$. The loss of symmetry in this example reflects  possible configurations in which agent ``$i$" ``sees''  agent ``$j$" but not the other way around. This example demonstrates a main difficulty in the flocking analysis of local influence functions, namely, counting the number of active agents $a_{ij}\geq \theta$ inside the cone of vision. We  leave the flocking analysis of this example to a future work.

\begin{figure}[ht]
  \centering
  \includegraphics[scale=1]{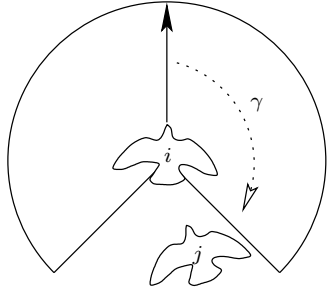}
  \caption{Adding a cone of vision in the C-S model (\ref{eq:CS_vision}) breaks down the symmetry of the interaction. Here, the agent ``$i$'' does not ``see'' the agent ``$j$'' whereas the agent ``$j$'' sees the agent ``$i$''.}
  \label{fig:vision}
\end{figure}


\end{document}